\documentclass[reqno]{amsart}

\usepackage{amsmath,amssymb,amsthm}

\usepackage{tikz}

\usepackage{caption}
\usepackage{graphicx}
\usepackage{graphics}

\newtheorem*{thm}{Theorem}
\newtheorem{lemma}{Lemma}

\newtheorem{corollary}{Corollary}

\newcommand{\lcm}{\operatorname{lcm}}

\begin{document}

\title[]{On the largest singular vector\\ of the Redheffer matrix}

\author[]{Fran\c{c}ois Cl\'ement}
\author[]{Stefan Steinerberger}
\address{Department of Mathematics, University of Washington, Seattle}
\email{fclement@uw.edu} 
\email{steinerb@uw.edu}

\begin{abstract}
The Redheffer matrix $A_n \in \mathbb{R}^{n \times n}$ is defined by setting $A_{ij} = 1$ if $j=1$ or $i$ divides $j$ and 0 otherwise. One of its many interesting properties is that $\det(A_n) = \mathcal{O}(n^{1/2 + \varepsilon})$ is equivalent to the Riemann hypothesis. The singular vector $v \in \mathbb{R}^n$ corresponding to the largest singular value carries a lot of information: $v_k$ is small if $k$ is prime and large if $k$ has many divisors. We prove that the vector $w$ whose $k-$th entry is the sum of the inverse divisors of $k$, $w_k = \sum_{d|k} 1/d$, is close to a singular vector in a precise quantitative sense.
\end{abstract}

 \maketitle

\section{Introduction and Result}
\subsection{Introduction}
One of the (mostly equivalent) definitions of the Redheffer matrix $A_n$, introduced by Redheffer \cite{redheffer} in 1977, is 
$$ A_{ij} = \begin{cases} 1 \qquad &\mbox{if}~j=1 ~\mbox{or if}~i|j \\ 0 \qquad &\mbox{otherwise.} \end{cases}$$
This means that every second entry in the second row is 1, every third entry in the third row and so on (see Fig. 1 for $n=8$ and $n=20$).
\begin{center}
\begin{figure}[h!]
        \begin{tikzpicture}
        \node at (0,0) {$\left(
\begin{array}{cccccccc}
 1 & 1 & 1 & 1 & 1 & 1 & 1 & 1 \\
 1 & 1 & 0 & 1 & 0 & 1 & 0 & 1 \\
 1 & 0 & 1 & 0 & 0 & 1 & 0 & 0 \\
 1 & 0 & 0 & 1 & 0 & 0 & 0 & 1 \\
 1 & 0 & 0 & 0 & 1 & 0 & 0 & 0 \\
 1 & 0 & 0 & 0 & 0 & 1 & 0 & 0 \\
 1 & 0 & 0 & 0 & 0 & 0 & 1 & 0 \\
 1 & 0 & 0 & 0 & 0 & 0 & 0 & 1 \\
\end{array}
\right)$};
\node at (5,0) {\includegraphics[width=0.27\textwidth]{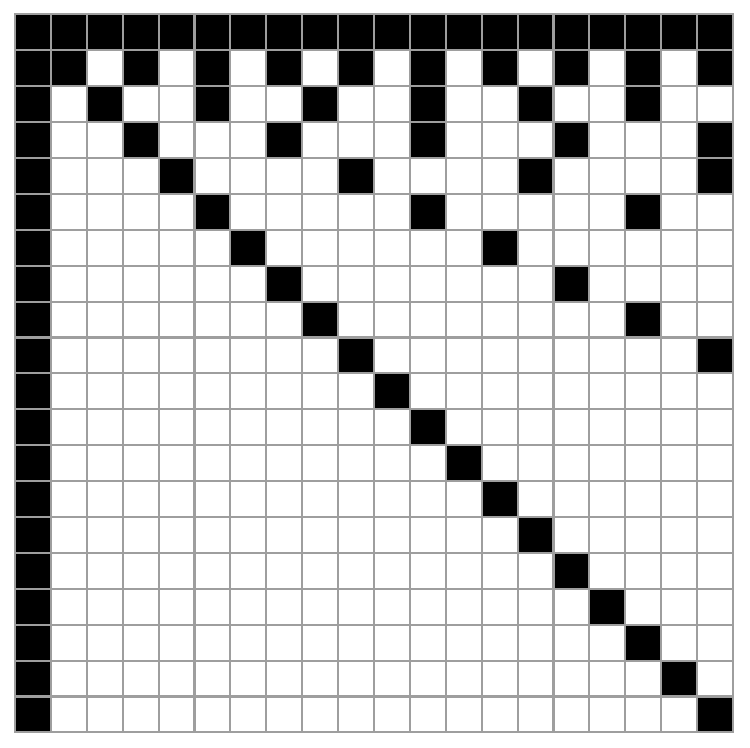}};
    \end{tikzpicture}
    \caption{Left: $A_{8}$. Right: $A_{20}$ with entries 1 highlighted.}
\end{figure}
\end{center}

Redheffer observed that
$$ \det(A_n) = \sum_{k=1}^{n} \mu(k),$$
 where $\mu: \mathbb{N} \rightarrow \left\{-1,0,1\right\}$ is the Moebius function. Littlewood \cite{littlewood} showed, see also Titchmarsh \cite[\S 14.25]{tit}, that the Riemann hypothesis is equivalent to the fact that for all $\varepsilon >0$, one has 
 $\det(A_n) = \mathcal{O}\left( n^{1/2 + \varepsilon}\right)$.
This suggests that understanding determinants of Redheffer matrices is perhaps not very easy. These matrices have attracted a lot of attention \cite{bord, pierre}, much of it focused on its spectral structure: we refer to
Barrett-Forcade-Pollington \cite{b0}, Barrett-Jarvis \cite{barrett}, Jarvis \cite{jarvis}, Kline \cite{k0, k1, k2}, Robinson-Barrett \cite{robinson}, Vaughan \cite{vaughan, vaughan2}. Less is known about the singular values, we refer to Cheon-Kim \cite{ch} and Hilberdink \cite{seventeen, hilberdink}.

\subsection{Two pictures.}
We were motivated by the following picture which shows the singular vector corresponding to the largest singular value of $A_{1000}$ on the left. The index $n=1000$ is somewhat arbitrary, the picture is consistent for all $A_n$ for all $n$ large. 
There is a `line' emerging at the bottom and this line corresponds to prime indices. There is a second line emerging that corresponds to indices that are of the form $2p$ with $p$ prime. On the other side, in the case of $n=1000$, the indices $i$ for which $v_i$ assumes a new record value are 2, 4, 6, 12, 24, 36, 48, 60, 120, 180, 240, 360, 720, 840 -- these are \textit{highly composite numbers}, numbers where the number of divisors reaches a new record, that are smaller than 1000.

\begin{center}
\begin{figure}[h!]
        \begin{tikzpicture}
\node at (0,0) {\includegraphics[width=0.48\textwidth]{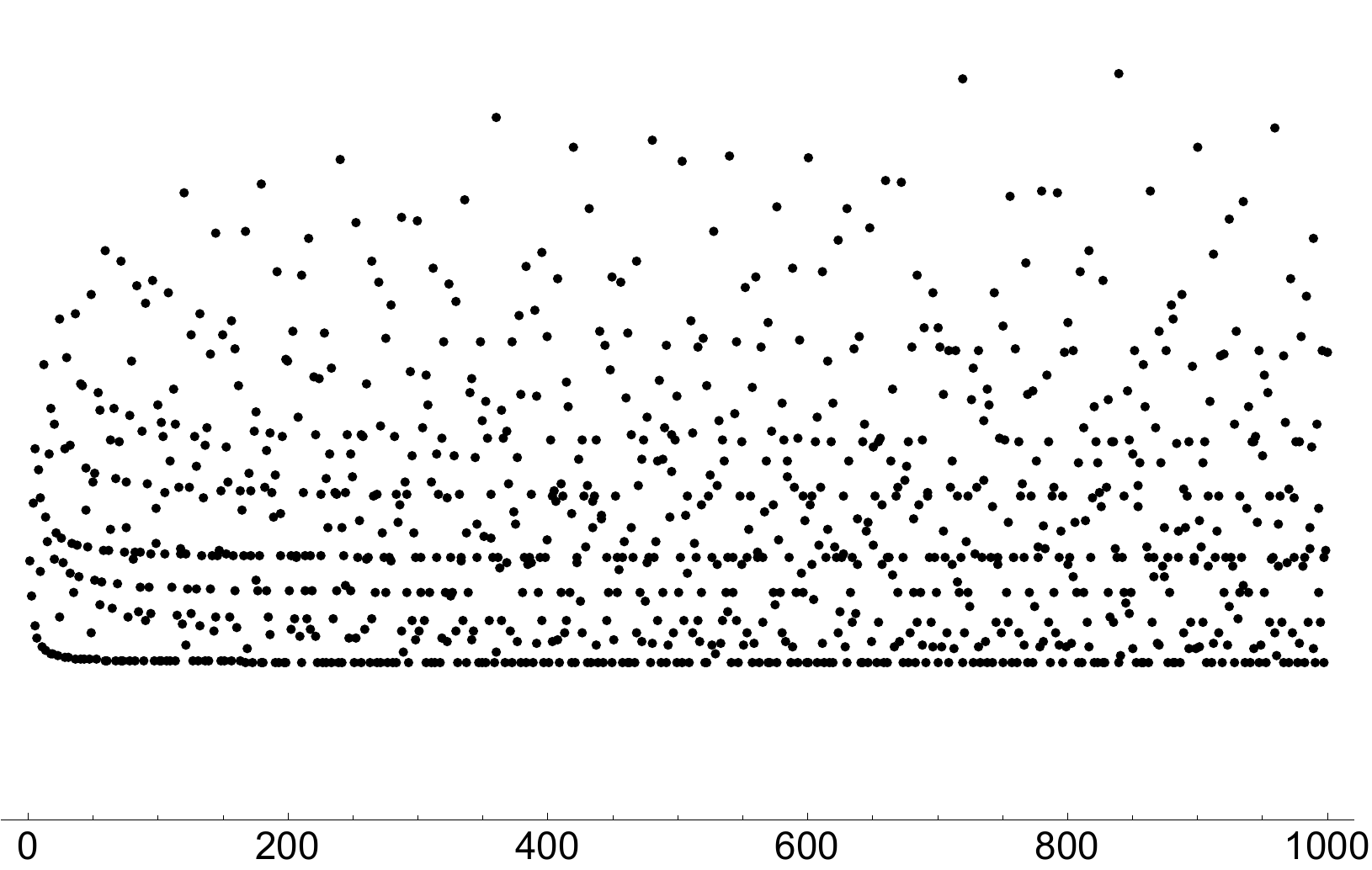}};
\node at (6.5,0) {\includegraphics[width=0.48\textwidth]{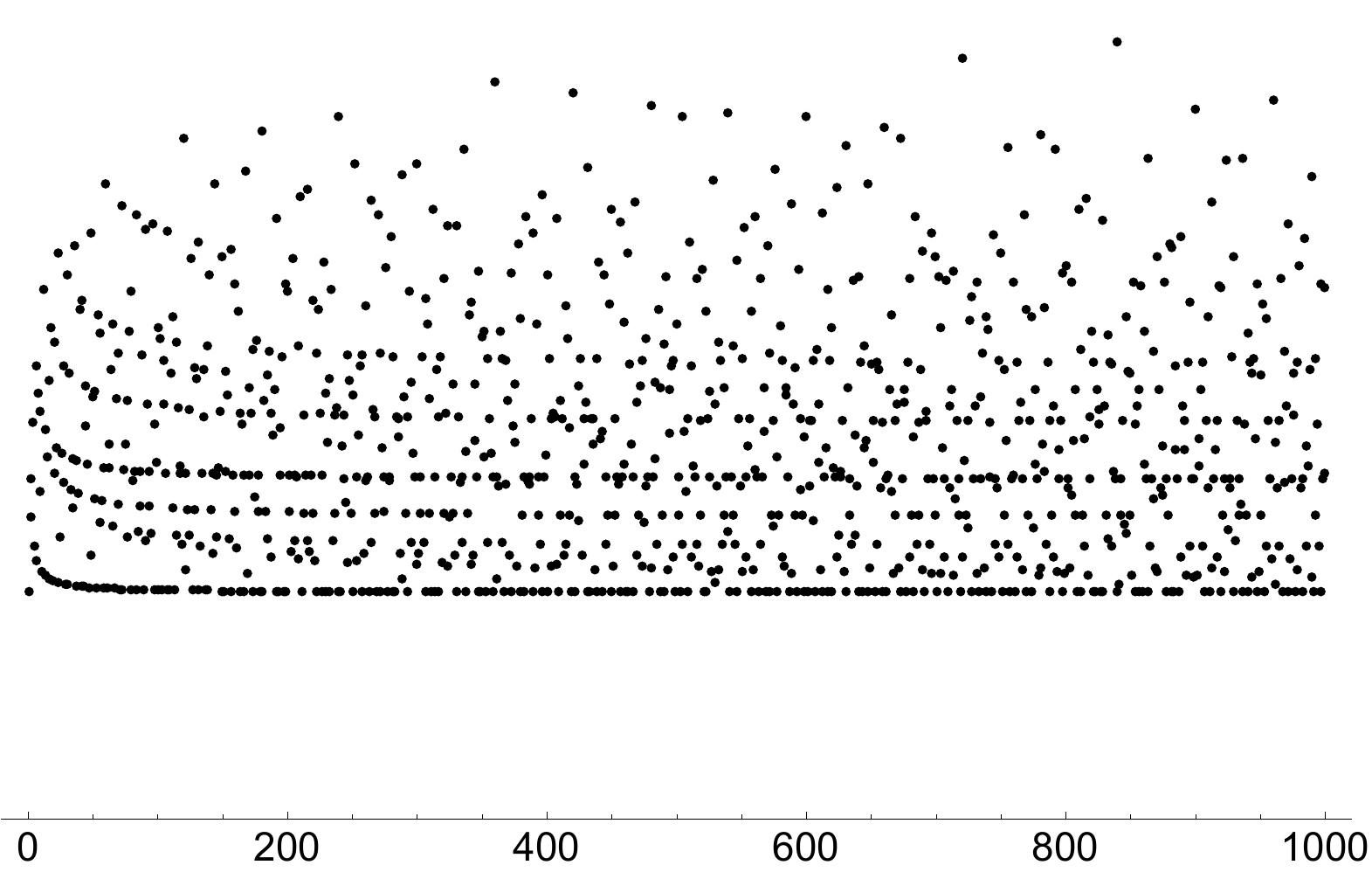}};
    \end{tikzpicture}
    \caption{Left: largest singular vector of $A_{1000}$.  `Lines' at the bottom correspond to primes, large entries to number with many divisors. Right: $v_{1000}$ looks almost like the singular vector.}
\end{figure}
\end{center}
The main purpose of our paper is to propose an approximation to what that singular vector may be (or rather: give a simple vector that seems to share much of the same characteristics).  Our candidate is the vector $v_n \in \mathbb{R}^n$ given by
$$ v_n =  \left(  \frac{\sigma(k)}{k} \right)_{k=1}^{n} = \left(\sum_{d|k} \frac{1}{d}\right)_{k=1}^{n},$$
where $\sigma(n)$ is the sum of all the divisors of $n$. Plotting this vector (see Fig. 3) shows that it appears to be \textit{very} similar to the largest singular vector; in particular, $(v_n)_k \sim 1$ when $k$ is a prime number. Conversely, $(v_n)_k$, as a function of $k$, assumes a new maximal value whenever $k$ is a `superabundant' number (see \cite{erd} and OEIS A004394). It is an interesting question whether the singular vectors assume their new records in highly composite numbers or superabundant numbers or something else entirely (the smallest highly composite number that is not superabundant is 7560, the smallest superabundant number that is not highly composite is 1163962800; these things may not be easy to test). A closer inspection shows that there are \textit{some} differences between the singular vector and $v$ but they are not easy to make precise.

\subsection{The Result.}

Our main result is a precise quantitative description showing that our candidate vector $v_n$ is `close' to a singular vector of $A_n$. To prepare our result, we quickly recall that singular vectors of a matrix $A_n$ are the same as eigenvectors of $A_n^T A_n$ and that the Cauchy-Schwarz inequality is only sharp if both vectors are linearly dependent. These two facts imply that
        $$ \left\langle \frac{v}{\|v\|}, \frac{A_n^T A_n v}{\| A_n^T A_n v\|} \right\rangle = 1 \quad \Longleftrightarrow \quad
 v~\mbox{is a singular vector of}~A_n.$$
Our main result is that, asymptotically, $v_n$ is not a singular vector of $A_n$ but it is, in this particular metric, \textit{very} close to being one!

\begin{thm} The limit
 $$ \lim_{n \rightarrow \infty} \left\langle \frac{v_n}{\|v_n\|}, \frac{A_n^T A_n v_n}{\| A_n^T A_n v_n\|} \right\rangle = \alpha \sim 0.9979\dots \qquad \mbox{exists.}$$
 Moreover, the limit $\alpha$ has the closed form expression
 $$ \alpha = \frac{\sqrt{2}}{\sqrt{5} \sqrt{\zeta(3)}} \left( \sum_{d=1}^{\infty} \frac{c_d^2}{d^2} \right)  \left( \sum_{ d_1, d_2=1}^{\infty}\frac{c_{d_1} c_{d_2} \gcd(d_1, d_2)}{d_1^2 d_2^2} \right)^{-1/2},$$
 where $c_d = \sum_{n=1}^{\infty} \gcd(n,d)/n^2$ and $\zeta$ is the Riemann zeta function.
\end{thm}
 Perhaps the most remarkable thing is (a) how close $v_n$ is to a singular vector and (b) that all the limits can actually be evaluated. The evaluation of the limits require some results from very classical analytical number theory including two identities stated by Ramanujan. It is an interesting question whether more can be said about the singular vector. We show, as part of the proof, that $A_n^T A_n$ can be decently well approximated by $(\sigma_0(\gcd(i,j)))_{i,j=1}^{n}$ with $\sigma_0(n)$ denoting the number of divisors of $n$. This type of matrix belongs to the class of generalized GCD matrices \cite{beslin, carlitz, haug, hong, jager, smith}. There, a phenomenon of this type if known: Aistleitner-Berkes-Seip-Weber \cite{aist} note that eigenvectors of suitable GCD matrices belonging to the maximal eigenvalue \textit{are concentrated on indices $k$ with many small prime factors} (with \cite{three, seventeen} being given as references). One might interpret our result as a way of precisely quantifying this phenomenon for $(\sigma_0(\gcd(i,j)))_{i,j=1}^{n}$.

\section{Proof}
The proof decouples into four different steps: it requires  
\begin{enumerate}
    \item understanding the behavior of $\|v_n\|$,
    \item simplifying $A_n^T A_n$ by replacing it with the matrix $B_n = (\sigma_0(\gcd(i,j)))_{i,j=1}^{n}$ and proving that the induced error is small enough for all subsequent steps,
    \item understanding the behavior of $\|B_n v\|$ 
    \item and understanding the behavior of the inner product $\left\langle v_n, B_n v_n \right\rangle$.
\end{enumerate}
We will derive all four parts in exactly this order.  In terms of notation, we will use $A \lesssim B$ if $A \leq cB$ for some universal constant $c>0$ (which is allowed to change its value from line to line). We use $A \gtrsim B$ in the same manner and $A \sim B$ will denote that both $A \lesssim B$ and $A \gtrsim B$. Norms $\|\cdot\|$ are always the Euclidean norm $\|\cdot\|_2$.

\subsection{Part 1. The size of $\|v_n\|$}
The first part of the argument follows from a stronger result of Ramanujan \cite{ramanujan}. Amusingly, $\zeta(3)$ makes an appearance.
\begin{lemma} \label{lem:1} We have
    $$\|v_n\|^2 =  \frac{5 \cdot \zeta(3) }{2} n + o(n).$$
\end{lemma}
\begin{proof}
 Ramanujan \cite{ramanujan} proved that
$$ \sum_{k\leq n} \sigma(k)^2 = \frac{5}{6} \zeta(3) n^3 + \mathcal{O}(n^2 \log^2{(n)}).$$
Smith \cite{smith} improved the error term to $\mathcal{O}(n^2 \log^{5/3}{(n)})$ but we do not require such precise estimates. For any $\varepsilon >0$  
$$ \sum_{(1-\varepsilon) n \leq k\leq n} \sigma(k)^2 = \frac{5}{6} \zeta(3) n^3 \left(1 - (1-\varepsilon)^3\right) + \mathcal{O}(n^2 \log^2{(n)}).$$
From this we deduce that
$$ X = \sum_{(1-\varepsilon) n \leq k\leq n} \frac{\sigma(k)^2}{k^2}$$
satisfies
$$   1 - (1-\varepsilon)^3 \leq  X   \cdot \frac{6}{5 \zeta(3)} \frac{1}{n}  + \mathcal{O}\left( \frac{\log^2{(n)}}{n}\right) \leq   \frac{\left(1 - (1-\varepsilon)^3\right)}{(1-\varepsilon)^2}.$$
A Taylor expansion shows
\begin{align*}
    1 - (1-\varepsilon)^3 &= 3 \varepsilon - 3 \varepsilon^2 + \mathcal{O}(\varepsilon^3) \\
    \frac{1 - (1-\varepsilon)^3}{(1-\varepsilon)^2} &= 3 \varepsilon + 3 \varepsilon^2 + \mathcal{O}(\varepsilon^3).
\end{align*}
This shows, for $n$ sufficiently large, that
$$ \frac{1}{\varepsilon n}  \sum_{(1-\varepsilon) n \leq k\leq n} \frac{\sigma(k)^2}{k^2} = \frac{5 \cdot \zeta(3)}{2} + \mathcal{O}(\varepsilon)$$
which implies the result.
\end{proof}

Lemma \ref{lem:1} implies a nice Corollary (which we will not use later).
\begin{corollary} We have
$$ \sum_{  d_1, d_2=1}^{\infty} \frac{\gcd(d_1, d_2)}{d_1^2 d_2^2} = \frac{5 \cdot \zeta(3)}{2}.$$
\end{corollary}
\begin{proof}
This follows from
    \begin{align*}
    \sum_{k=1}^{n} \frac{\sigma(k)^2}{k^2}  &=  \sum_{k=1}^{n} \left(\sum_{d |k} \frac{1}{d} \right)^2 =  \sum_{k=1}^{n} \sum_{d_1|k, d_2 |k} \frac{1}{d_1 d_2} \\
        &= \sum_{1 \leq d_1, d_2 \leq n} \frac{1}{d_1 d_2} \left\lfloor \frac{n}{\lcm(d_1, d_2)} \right\rfloor\\
        &= n\sum_{1 \leq d_1, d_2 \leq n} \frac{1}{d_1 d_2}  \frac{1}{\lcm(d_1, d_2)} + \mathcal{O}\left( \sum_{1 \leq d_1, d_2 \leq n} \frac{1}{d_1 d_2} \right).
    \end{align*}
Using
$$ \frac{1}{\lcm(d_1, d_2)} = \frac{\gcd(d_1, d_2)}{d_1 d_2}$$
    together with
    $$  \sum_{1 \leq d_1, d_2 \leq n} \frac{1}{d_1 d_2} = \left( \sum_{d=1}^{n} \frac{1}{d} \right)^2 \lesssim \log^2{n}$$
and letting $n \rightarrow \infty$ implies the statement.\end{proof}

\subsection{Part 2. Simplifying $A_n^T A_n$.}
We first deduce a formula for the entries of $A_n^T A_n$. $\sigma_0(n)$ will denote the number of divisors of $n$. The explicit formula can be used to argue that the simplified matrix $B_n = (\sigma_0(\gcd(i,j)))_{i,j=1}^{n}$ is close enough to the original matrix to continue the analysis.
\begin{lemma} \label{lem:2} The matrix $ A_n^T A_n$ satisfies
$$ ( A_n^T A_n)_{ij} = \begin{cases} n \qquad &\mbox{if}~(i,j)= (1,1) \\
\sigma_0(j) \qquad &\mbox{if} ~ j > i  = 1\\
\sigma_0(i) \qquad &\mbox{if} ~ i > j  = 1\\
\sigma_0(\gcd(i,j)) \qquad &\mbox{otherwise.}
\end{cases}$$
\end{lemma}
\begin{proof}
Let $A_{i,j}$ be the $(i,j)$ entry of $A_n$. By matrix transposition, we have that $$( A_n^T A_n)_{ij}=\sum_{k=1}^n (A^T)_{i,k} A_{k,j} = \sum_{k=1}^n A_{k,i} A_{k,j}.$$
  If $i=j=1$, then $A_{k,i}=A_{k,j}=1$ for any value of $k$ and the sum is $n$. If $j>i=1$, $A_{k,i}=1$, and $A_{k,j}=1$ if and only if $k|j$ which happens $\sigma_0(j)$ many times. By symmetry, we also get the third case $i>j=1$. In the last case, for $i,j>1$, for any $k\in \{1,\ldots,n\}$, since the entries of $A$ are either 0 or 1, we have that $A_{k,i}A_{k,j}=1$ if and only if $A_{k,i}=1$ and $A_{k,j}=1$. By definition, this is equivalent to $k|i$ and $k|j$ and thus $k|\gcd(i,j)$ showing that $( A_n^T A_n)_{ij}$ is $\sigma_0(\gcd(i,j))$.
\end{proof}

This naturally suggests to focus on the simpler matrix
$$B_{n} = \left( \sigma_0(\gcd(i,j)\right)_{i,j=1}^{n}.$$
Lemma \ref{lem:3} shows that $B_n v_n$ and $A_n^T A_n v_n$ are fairly close.
\begin{lemma} \label{lem:3}
For all $1 \leq i \leq n$
$$ [(A_n^T A_n - B_n)v_n]_i \lesssim \begin{cases} 
n+\sum_{j=2}^n \frac{\sigma_0(j)\sigma(j)}{j} \qquad &\mbox{if}~ i=1\\
\sigma_0(i) \qquad &\mbox{otherwise.}
\end{cases}$$
Moreover, for any $\varepsilon>0$ there exist $C_{\varepsilon}$ so that for all $n \in \mathbb{N}$
    $$ \| (A_n^T A_n - B_n)v_n \| \leq C_{\varepsilon} n^{1+\varepsilon}.$$
\end{lemma}
\begin{proof}
Let $x=(A_n^T A_n - B_n)v_n$. Lemma \ref{lem:2} shows that the matrix $A_n^T A_n - B_n$ has nonzero entries only in the first row and the first column. The first row shows that when $2 \leq i \leq n$ we have $x_i=\sigma_0(i)-1$.
A computation shows
\begin{align*}
    x_1 &= n-1  + \sum_{j=2}^{n} (\sigma_0(j) - 1) \frac{\sigma(j)}{j}.
\end{align*}
 This shows the first part of the statement. For the second part of the statement, we argue as follows:
for any $\varepsilon > 0$ we have $\sigma_0(n) = o(n^{\varepsilon})$. Combining this with a standard asymptotic (a refined version is proved in Lemma \ref{lem:5}), we have
\begin{align*}
    \sum_{j=2}^{n} (\sigma_0(j) - 1) \frac{\sigma(j)}{j} &= o(n^{\varepsilon}) \sum_{j=1}^{n}  \frac{\sigma(j)}{j} = o(n^{\varepsilon}) \left( \frac{\pi^2}{6} n + \log{(n)}\right).
\end{align*}
This shows that the contribution of the first row is $\mathcal{O}_{\varepsilon}(n^{1+\varepsilon})$. It remains to show that the contribution from all the other rows is small. Here, we use yet another formula of Ramanujan (proven by Wilson \cite{wilson})
$$ \sum_{i=1}^{n} \sigma_0(i)^2 = \frac{1 + o(1)}{\pi^2} \cdot n\cdot (\log{n})^3 = \mathcal{O}_{\varepsilon}(n^{1+\varepsilon}). $$
\end{proof}

\subsection{Part 3. Understanding $B_n v_n$.}
Looking at a single row, we have
$$ (B_n v_n)_i = \sum_{k=1}^{n} \sigma_0(\gcd(i,k)) \frac{\sigma_1(k)}{k}.$$
The first observation is that this term can be rewritten.
\begin{lemma} \label{lem:4}
    We have, for any $1 \leq i \leq n$,
 $$    \sum_{k=1}^{n} \sigma_0(\gcd(i,k)) \frac{\sigma_1(k)}{k} = \sum_{d|i} \sum_{k=1}^{\left\lfloor n/d \right\rfloor}  \frac{\sigma_1(k d)}{k d}$$
\end{lemma}
\begin{proof} A number divides $\gcd(i,k)$ if and only if it both divides $i$ and $k$, therefore
  $$\sum_{k=1}^{n} \sigma_0(\gcd(i,k)) \frac{\sigma_1(k)}{k}=\sum_{k=1}^{n} \left(\sum_{d|i ~\mbox{\tiny and}~ d|k} 1\right) \frac{\sigma_1(k)}{k}.$$
 Recalling that $i$ is fixed, we may rewrite the expression as
$$ \sum_{d|i ~\mbox{\tiny and}~ d|k} 1 = \sum_{d|i} \mathbf{1}_{d|k}.$$
Changing the order of summation,
 \begin{align*}
     \sum_{k=1}^{n} \left(\sum_{d|i ~\mbox{\tiny and}~ d|k} 1\right) \frac{\sigma_1(k)}{k}&= \sum_{k=1}^{n} \sum_{d|i } \mathbf{1}_{d|k}  \frac{\sigma_1(k)}{k}\\
     &=\sum_{d|i }\sum_{k=1}^{n} \mathbf{1}_{d|k}  \frac{\sigma_1(k)}{k}= \sum_{d|i} \sum_{j=1}^{\lfloor n/d \rfloor}\frac{\sigma_1(j d)}{j d}.
 \end{align*}
\end{proof}
The next step consists of finding an asymptotic result for the expression that naturally arose in Lemma \ref{lem:4}. The special case $\ell =1$ has already been used above in the proof of Lemma \ref{lem:3}.
Lemma \ref{lem:5} is by no means original: the case $\ell = 1$ is completely classical with more precise results for the error term having been given by Walfisz \cite{wal} and P\'etermann \cite{peter0, peter1, peter2}.
We have not been able to find the case $\ell > 1$ in the literature but the usual elementary argument works.

\begin{lemma} \label{lem:5}
    We have, for all $\ell,n \in \mathbb{N}$, that
    $$  \sum_{k=1}^{n} \frac{\sigma_1(\ell k)}{\ell k}  = n \sum_{d=1}^{\infty}  \frac{ \gcd(d,\ell)}{d^2} + \mathcal{O}(\log(\ell n)).$$
\end{lemma}
\begin{proof} We start by rewriting
 \begin{align*}
      \sum_{k=1}^{n} \frac{\sigma_1(\ell k)}{\ell k} = \sum_{k=1}^{n} \sum_{d|\ell k} \frac{1}{d} =\sum_{d=1}^{\ell n} \frac{1}{d} \cdot \# \left\{ k \leq n: d|\ell k \right\}
 \end{align*}
 $d$ divides $\ell k$ if and only if $k$ is a multiple of $d/\gcd(d, \ell)$. Therefore
 $$ \# \left\{ k \leq n: d|\ell k \right\} = \left\lfloor \frac{n \gcd(d,\ell)}{d} \right\rfloor $$
and thus
 \begin{align*}
      \sum_{k=1}^{n} \frac{\sigma_1(\ell k)}{\ell k} &=  \sum_{d=1}^{\ell n} \frac{1}{d} \cdot  \left\lfloor \frac{n \gcd(d,\ell)}{d} \right\rfloor = \sum_{d=1}^{\ell n} \frac{1}{d} \cdot  \frac{n \gcd(d,\ell)}{d} + \mathcal{O}\left( \sum_{d=1}^{\ell n} \frac{1}{d}\right) \\
      &= n \sum_{d=1}^{\ell n}  \frac{ \gcd(d,\ell)}{d^2} + \mathcal{O}(\log(\ell n)).
 \end{align*}
The sum can be made unbounded since
\begin{align*}
   n \sum_{d=1}^{\ell n}  \frac{ \gcd(d,\ell)}{d^2} =  n \sum_{d=1}^{\infty}  \frac{ \gcd(d,\ell)}{d^2} + \mathcal{O}\left(  n \sum_{d=\ell n +1}^{\infty }  \frac{ \gcd(d,\ell)}{d^2} \right)
\end{align*}
where the trivial estimate $\gcd(d,\ell) \leq \ell$, shows that the error is small
$$ n \sum_{d=\ell n +1}^{\infty }  \frac{ \gcd(d,\ell)}{d^2}  \leq \ell n\sum_{d=\ell n +1}^{\infty }  \frac{1}{d^2}  \lesssim 1.$$
\end{proof}

Lemma \ref{lem:5} motivates introducing the abbreviation
$$ c_{\ell} = \sum_{d=1}^{\infty}  \frac{ \gcd(d,\ell)}{d^2}.$$
This object will show up in many of the subsequent arguments (as well as in the statement of our main result).
There seems to be a folklore identity
$$ c_{\ell} = \sum_{d=1}^{\infty}  \frac{ \gcd(d,\ell)}{d^2} =\frac{\zeta(2)}{\ell^2} \sum_{d|\ell} d^2 \phi\left( \frac{\ell}{d} \right) = \frac{\zeta(2)}{\ell^2}\sum_{d=1}^{\ell} \gcd(\ell, d)^2
$$
whose origin we were unable to trace (and which is not needed for subsequent steps); the equivalence of the second and the third term follows from an old result of C\'esaro \cite{cesaro}.
We note that $c_{\ell} \geq 1$. An easy bound shows that it is never very large
\begin{align*}
    c_{\ell} = \sum_{d=1}^{\infty}  \frac{ \gcd(d,\ell)}{d^2} &= \sum_{d \leq \ell}^{}  \frac{ \gcd(d,\ell)}{d^2} + \sum_{d = \ell +1}^{\infty}  \frac{ \gcd(d,\ell)}{d^2} \\
    &\leq \sum_{d \leq \ell}^{}  \frac{d}{d^2}  +  \sum_{d = \ell +1}^{\infty}  \frac{ \ell}{d^2} \lesssim \log{\ell}.
\end{align*}
However, this easy estimate is clearly pessimistic and the subsequent arguments will never require a point-wise estimate. On average, $c_{\ell}$ is $\sim 1$.

\begin{lemma} \label{lem:dyadic}
There exists a constant $c>0$ such that for all $X \geq 1$
    $$ \sum_{X \leq \ell \leq 2X} c_{\ell} \leq c X.$$
\end{lemma}
\begin{proof}
    Exchanging the order of summation, we have
    \begin{align*}
\sum_{X \leq \ell \leq 2X} c_{\ell} &= \sum_{X \leq \ell \leq 2X} \sum_{d=1}^{\infty}  \frac{ \gcd(d,\ell)}{d^2} = \sum_{d=1}^{\infty}  \frac{1}{d^2} \sum_{X \leq \ell \leq 2X}  \gcd(d,\ell) \\
&= \sum_{d = 1} ^{X}  \frac{1}{d^2} \sum_{X \leq \ell \leq 2X}  \gcd(d,\ell) + \sum_{d = X + 1}^{\infty}  \frac{1}{d^2} \sum_{X \leq \ell \leq 2X}  \gcd(d,\ell).
    \end{align*}
Using the trivial bound $\gcd(d,\ell) \leq \ell$ shows that the second sum is at most linear
    \begin{align*}
\sum_{X \leq \ell \leq 2X} c_{\ell} &\leq \sum_{d =1 } ^{X}  \frac{1}{d^2} \sum_{X \leq \ell \leq 2X }  \gcd(d,\ell) + \sum_{d = X+1}^{\infty}  \frac{X \cdot 2X}{d^2} \\
&\lesssim \sum_{d =1}^{X}  \frac{1}{d^2} \sum_{X \leq \ell \leq 2X}  \gcd(d,\ell) + X.
    \end{align*}

    It remains to bound the first sum. For fixed $d \in \mathbb{N}$, the sequence $\ell \rightarrow \gcd(d,\ell)$ is $d-$periodic. Moreover, the sum over one period is Pillai's \cite{pillai}  arithmetic function $P(d)$ for which an identity of  C\'esaro \cite{cesaro} implies  that
    $$ P(d) = \sum_{k=1}^{d} \gcd(k,d) = \sum_{e | d} e \cdot \phi\left( \frac{d}{e} \right),$$
    where $\phi$ is the Euler totient function.
Crude estimates ($\sigma_{0}(d) = \mathcal{O}_{\varepsilon}(d^{\varepsilon})$ and $\phi(n) \leq n$) show that $P(d)$ is not much larger than linear
$$ P(d) = \sum_{e | d} e \cdot \phi\left( \frac{d}{e} \right) \leq \sum_{e | d} d = \mathcal{O}_{\varepsilon}(d^{1+\varepsilon}).$$
Therefore, since we sum over $\lesssim X/d + 1$ such periods,
\begin{align*}
    \sum_{d = 1}^{X}  \frac{1}{d^2} \sum_{X \leq \ell \leq 2X}  \gcd(d,\ell) &\leq     \sum_{d = 1}^{X}  \frac{1}{d^2}  \left(\frac{X}{d} + 1\right)P(d) \\
    &\leq X \sum_{d =1 }^{X}  \frac{P(d)}{d^3} + \sum_{d =1}^{X} \frac{P(d)}{d^2} \\
    &\lesssim X + \sum_{d=1}^{X} \frac{d^{3/2}}{d^2} \lesssim X + \sqrt{X} \lesssim X.
\end{align*}
\end{proof}
This estimate immediately implies estimates along the lines of
$$ \sum_{\ell = 1}^{X} \frac{c_{\ell}}{\ell} \lesssim \log{X} \qquad \mbox{and} \qquad  \sum_{\ell = 1}^{X} \frac{c_{\ell}}{\ell^2} \lesssim 1$$
by splitting the summation interval into dyadic intervals and repeatedly applying Lemma \ref{lem:dyadic}. This will frequently be used in the last part of the argument. Combining Lemma \ref{lem:4} and Lemma \ref{lem:5} now allows us to conclude the main result of this subsection.
\begin{lemma} \label{lem:6} For any $\varepsilon >0$ and all $1 \leq i \leq n$,
 $$  (B_n v_n)_{i} =  \sum_{k=1}^{n} \sigma_0(\gcd(i,k)) \frac{\sigma_1(k)}{k} =  n \sum_{d|i} \frac{c_d}{d} + \mathcal{O}_{\varepsilon}(n^{\varepsilon}).$$
\end{lemma}
\begin{proof}
Lemma \ref{lem:4} and Lemma \ref{lem:5} give
\begin{align*}
         \sum_{k=1}^{n} \sigma_0(\gcd(i,k)) \frac{\sigma_1(k)}{k} &= \sum_{d|i} \sum_{k=1}^{\left\lfloor n/d \right\rfloor}  \frac{\sigma_1(k d)}{k d} =\sum_{d|i} \left( c_d \left\lfloor \frac{n}{d}\right\rfloor + \mathcal{O}(\log(n)) \right).
\end{align*}
We have
\begin{align*}
    \sum_{d|i} c_d \left\lfloor \frac{n}{d}\right\rfloor &=   \sum_{d|i} c_d  \frac{n}{d} + \mathcal{O}\left(   \sum_{d|i} c_d \rfloor\right) = n\sum_{d|i}   \frac{c_d}{d} +  \mathcal{O}\left(   \sum_{d|i} \log{d}  \right) \\
    &= n\sum_{d|i}   \frac{c_d}{d} +  \mathcal{O}\left(  \sigma_0(i) \log{i} \right)
\end{align*}
Since each number $k \leq n$ has $ \sigma_0(k)\leq c_{\varepsilon} n^{\varepsilon}$ divisors, this error term is $\mathcal{O}_{\varepsilon}\left( n^{\varepsilon}\right)$ and the same is true for the other error term since
$$ \sum_{d|i} \mathcal{O}(\log(n)) = \mathcal{O}_{\varepsilon}\left( n^{\varepsilon/2} \log{n}\right) = \mathcal{O}_{\varepsilon}\left( n^{\varepsilon}\right).$$
\end{proof}

\subsection{Norm of $\|B_n v_n\|$} The asymptotic expansion derived in the previous section for the behavior of a single row can now be combined to derive an asymptotic expression for the asymptotic behavior of $\|B_n v_n\|$. 
\begin{lemma} \label{lem:8}
We have
$$ \| B_n v_n\| = n^{3/2} \left( \sum_{ d_1, d_2=1}^{\infty}\frac{c_{d_1} c_{d_2} \gcd(d_1, d_2)}{d_1^2 d_2^2} \right)^{1/2} + \mathcal{O}_{\varepsilon}\left( n^{1/2 + \varepsilon}\right).$$
\end{lemma}
\begin{proof}
Lemma \ref{lem:6}  showed
$$  (B_n v_n)_{i} =  n \sum_{d|i} \frac{c_d}{d} + \mathcal{O}_{\varepsilon}(n^{\varepsilon}).$$
Using the crude bound $c_{\ell} \lesssim \log{\ell}$ from above,
$$ \sum_{d|i} \frac{c_d}{d} \lesssim \sum_{d|i} \frac{1+\log{d}}{d} \leq \sum_{d=1}^{i} \frac{1 + \log{d}}{d} \lesssim \log^2{(i)} \leq \log^2{(n)} $$
we deduce
$$ (B_n v_n)_{i}^2 = n^2 \left(\sum_{d|i} \frac{c_d}{d} \right)^2 + \mathcal{O}_{\varepsilon}(n^{1+\varepsilon}) $$
which implies, by summation over $i$,
$$ \| B_n v_n\|^2 = \sum_{i=1}^{n} (B_n v_n)_{i}^2  = n^2 \sum_{i=1}^{n}  \left(\sum_{d|i} \frac{c_d}{d}\right)^2 + \mathcal{O}_{\varepsilon}(n^{2+\varepsilon}).$$
Exchanging the order of summation,
\begin{align*}
     n^2\sum_{i=1}^{n}  \left(\sum_{d|i} \frac{c_d}{d}\right)^2 &=       n^2\sum_{i=1}^{n}  \sum_{d_1|i, d_2|i} \frac{c_{d_1} c_{d_2}}{d_1 d_2} =  n^2\sum_{d_1, d_2=1}^{n} \frac{c_{d_1} c_{d_2}}{d_1 d_2} \left\lfloor \frac{n}{\lcm(d_1, d_2)} \right\rfloor \\
     &=  n^2\sum_{d_1, d_2 = 1}^{n} \frac{c_{d_1} c_{d_2}}{d_1 d_2} \frac{n}{\lcm(d_1, d_2)} + \mathcal{O} \left(  n^2\sum_{d_1, d_2=1}^{n} \frac{c_{d_1} c_{d_2}}{d_1 d_2}\right).
\end{align*}
We start with the main term. Using $\gcd(d_1, d_2) \lcm(d_1, d_2) = d_1 d_2$, it can be rewritten as
\begin{align*}
  n^2\sum_{d_1, d_2=1}^{n}  \frac{c_{d_1} c_{d_2}}{d_1 d_2} \frac{n}{\lcm(d_1, d_2)}  &=   n^2\sum_{d_1, d_2=1}^{n}  \frac{c_{d_1} c_{d_2}}{d_1 d_2} \frac{n \gcd(d_1, d_2)}{d_1 d_2} \\
 &= n^3  \sum_{d_1, d_2=1}^{n} \frac{c_{d_1} c_{d_2} \gcd(d_1, d_2)}{d_1^2 d_2^2}.
\end{align*}
As usual, we would like to argue that we can replace the finite double sum by the infinite double sum. Since the summation indices and the summand are symmetric, we can bound this error from above by
\begin{align*}
    2 \sum_{d_1 = n+1}^{\infty} \sum_{d_2=1}^{\infty}  \frac{c_{d_1} c_{d_2} \gcd(d_1, d_2)}{d_1^2 d_2^2} &=     2 \sum_{d_1 = n+1}^{\infty}  \frac{c_{d_1}}{d_1^2} \sum_{d_2=1}^{\infty}  \frac{  c_{d_2} \gcd(d_1, d_2)}{  d_2^2}.
\end{align*}
Using the trivial bound $\gcd(a,b) \leq \min\left\{a,b\right\}$ and a splitting of the sum into dyadic intervals together with Lemma \ref{lem:dyadic} gives
\begin{align*}
     \sum_{d_2=1}^{\infty}  \frac{  c_{d_2} \gcd(d_1, d_2)}{  d_2^2} &= \sum_{d_2 \leq d_1}^{}  \frac{  c_{d_2} \gcd(d_1, d_2)}{  d_2^2} + \sum_{d_2 = d_1 + 1}^{\infty}  \frac{  c_{d_2} \gcd(d_1, d_2)}{  d_2^2} \\
     &\leq \sum_{d_2 \leq d_1}^{}  \frac{  c_{d_2} }{ d_2} + \sum_{d_2 = d_1 + 1}^{\infty}  \frac{  c_{d_2} d_1}{  d_2^2} \lesssim \log(d_1).
\end{align*}
Plugging this back into the first sum and using
$$ \int_{e}^{\infty} \frac{\log{(x)}}{x^2} dx = \frac{1 + \log{e}}{e},$$
we 
get 
$$  \sum_{d_1 = n+1}^{\infty} \sum_{d_2=1}^{\infty}  \frac{c_{d_1} c_{d_2} \gcd(d_1, d_2)}{d_1^2 d_2^2} \lesssim \frac{\log{n}}{n}.$$
This shows that the error of extending the summation in the main term from $n$ to $\infty$ is $\mathcal{O}(n^2 \log{n})$ It remains to deal with the error term which is easy: dyadic summation and Lemma \ref{lem:dyadic} imply

\begin{align*}
\sum_{d_1, d_2=1}^{n} \frac{c_{d_1} c_{d_2}}{d_1 d_2} = \left(\sum_{d_1=1}^{n}  \frac{c_{d_1} }{d_1}\right)^2 \lesssim \log^2{(n)}.
\end{align*}
This implies
$$ \| B_n v_n\|^2 = n^{3} \sum_{ d_1, d_2=1}^{\infty}\frac{c_{d_1} c_{d_2} \gcd(d_1, d_2)}{d_1^2 d_2^2} + \mathcal{O}_{\varepsilon}(n^{2+\varepsilon}).$$
Finally, using the mean value theorem in the form
$$ \sqrt{c n^3 + y} - \sqrt{c n^3} \sim \frac{y}{2 \sqrt{c n^{3/2}}}$$
for $0 \leq y \ll n^3$ gives the desired result.
\end{proof}

\subsection{Norm of $\left\langle v_n, B_n v_n\right\rangle$}
The purpose of this section is to combine the above arguments to derive an asymptotic expansion for the inner product $\left\langle v_n, B_n v_n\right\rangle$.
\begin{lemma} \label{lem9} We have, for all $\varepsilon > 0$,
$$ \left\langle v_n, B_n v_n\right\rangle = n^2  \sum_{d=1}^{\infty} \frac{c_d^2}{d^2} + \mathcal{O}_{\varepsilon}(n^{1+ \varepsilon}).$$    
\end{lemma}
\begin{proof}
Lemma \ref{lem:6} and Lemma \ref{lem:5} with $\ell=1$ for the error term imply
    \begin{align*}
        \left\langle v_n, B_n v_n\right\rangle &= \sum_{i=1}^{n} \left(n \sum_{d|i} \frac{c_d}{d} + \mathcal{O}_{\varepsilon}(n^{\varepsilon}) \right) \frac{\sigma_1(i)}{i} \\
        &=   n\sum_{i=1}^{n} \left(   \sum_{d|i}   \frac{c_d}{d} \right) \left(   \sum_{d|i}   \frac{1}{d} \right) + \mathcal{O}_{\varepsilon}(n^{1+ \varepsilon}) = n\sum_{i=1}^{n} \sum_{d_1 | i \atop d_2 |i } \frac{c_{d_1}}{d_1 d_2} + \mathcal{O}_{\varepsilon}(n^{1+ \varepsilon}).
    \end{align*}
Exchanging the order of summation gives
\begin{align*}
 n\sum_{i=1}^{n} \sum_{d_1 | i \atop d_2 |i } \frac{c_{d_1}}{d_1 d_2}  &= n\sum_{d_1, d_2 =1}^{n} \frac{c_{d_1}}{d_1 d_2} \left\lfloor \frac{n}{\lcm(d_1,d_2)} \right\rfloor \\
 &= n\sum_{d_1, d_2 =1}^{n} \frac{c_{d_1}}{d_1 d_2}   \frac{n}{\lcm(d_1,d_2)}  + \mathcal{O}\left( n\sum_{d_1, d_2 =1}^{n} \frac{c_{d_1}}{d_1 d_2}  \right) \\
 &= n^2  \sum_{d_1, d_2 =1}^{n} \frac{c_{d_1} \gcd(d_1, d_2)}{d_1^2 d_2^2 }    + \mathcal{O}\left(n \sum_{d_1, d_2 =1}^{n} \frac{c_{d_1}}{d_1 d_2}  \right).
\end{align*}
We start by bounding the error term, using $c_{\ell} \geq 1$,
\begin{align*}
   n \sum_{d_1, d_2 =1}^{n} \frac{c_{d_1}}{d_1 d_2} &\leq n \left(\sum_{d_1 = 1}^{n} \frac{c_{d_1}}{d_1}\right)^2 \lesssim n\log^2{(n)}.
\end{align*}
Regarding the main term, as before, we extend the summation in the main term from $n$ to $\infty$. We note that we already bounded the error for a sum with a larger summand (because $c_{d_2} \geq 1$) in the proof of Lemma \ref{lem:8} where we showed that
$$  \sum_{d_1 = n+1}^{\infty} \sum_{d_2=1}^{\infty}  \frac{c_{d_1} c_{d_2} \gcd(d_1, d_2)}{d_1^2 d_2^2} \lesssim \frac{\log{n}}{n}.$$
Therefore
$$ \left\langle v_n, B_n v_n\right\rangle = n  \sum_{d_1, d_2 =1}^{\infty} \frac{c_{d_1} \gcd(d_1, d_2)}{d_1^2 d_2^2 } + \mathcal{O}_{\varepsilon}(n^{1+ \varepsilon}).$$
The main term can be simplified a little bit more
\begin{align*}
     \sum_{d_1= 1}^{\infty} \sum_{d_2 = 1}^{\infty} \frac{c_{d_1} \gcd(d_1, d_2)}{d_1^2 d_2^2} &= \sum_{d_1= 1}^{\infty} \frac{c_{d_1}}{d_1^2} \sum_{d_2 = 1}^{\infty} \frac{ \gcd(d_1, d_2)}{ d_2^2} \\
     &=  \sum_{d_1= 1}^{\infty} \frac{c_{d_1}}{d_1^2}  c_{d_1} = \sum_{d=1}^{\infty} \frac{c_d^2}{d^2}.
\end{align*}
\end{proof}

\subsection{Proof of Theorem}
\begin{proof} We can now collect all the results: we have
$$      \| (A_n^T A_n - B_n)v_n \| \leq C_{\varepsilon} n^{1+\varepsilon}.$$
Therefore, using Cauchy-Schwarz, we have
\begin{align*}
\left| \left\langle v_n, A_n^T A_n  v_n \right\rangle  -  \left\langle v_n, B_n v_n \right\rangle \right| =\left| \left\langle v_n, (A_n^T A_n -B_n) v_n \right\rangle \right| \lesssim \|v_n\| \cdot n^{1+\varepsilon}.
\end{align*}
Therefore
\begin{align*}
     \left\langle \frac{v_n}{\|v_n\|}, \frac{A_n^T A_n v_n}{\| A_n^T A_n v_n\|} \right\rangle &= 
      \frac{\left\langle v_n, A_n^T A_n v_n \right\rangle}{\|v_n\| \cdot \| A_n^T A_n  v_n\|}\\
      &=
     \frac{\left\langle v_n, B_n v_n \right\rangle + \mathcal{O}_{\varepsilon}( \|v_n\| n^{1+\varepsilon})}{\|v_n\| \cdot \left(\|B_n v_n\| +\mathcal{O}_{\varepsilon}( \|v_n\| n^{1+\varepsilon}) \right) }.
\end{align*}
At this point, we start invoking asymptotics. We have
\begin{align*}
    \left\langle v_n, B_n v_n \right\rangle &= n^2  \sum_{d=1}^{\infty} \frac{c_d^2}{d^2} + \mathcal{O}_{\varepsilon}(n^{1+ \varepsilon}) \\
    \|v_n\| &=  \left(\frac{5}{2} \zeta(3) \cdot n\right)^{1/2} + o(\sqrt{n}) \\
    \|B_n v_n\| &= n^{3/2} \left( \sum_{ d_1, d_2=1}^{\infty}\frac{c_{d_1} c_{d_2} \gcd(d_1, d_2)}{d_1^2 d_2^2} \right)^{1/2} + \mathcal{O}_{\varepsilon}\left(n^{1/2 + \varepsilon}\right).
\end{align*}
This shows that
$$   \left\langle \frac{v_n}{\|v_n\|}, \frac{A_n^T A_n v_n}{\| A_n^T A_n v_n\|} \right\rangle  =  \frac{\left\langle v_n, B_n v_n \right\rangle}{\|v_n\| \cdot \|B_n v_n\|} + o(1)$$
and, by computing the constants,
$$ \lim_{n \rightarrow \infty} \frac{\left\langle v_n, B_n v_n \right\rangle}{\|v_n\| \cdot \|B_n v_n\|} =
 \frac{\sqrt{2}}{\sqrt{5} \sqrt{\zeta(3)}} \left( \sum_{d=1}^{\infty} \frac{c_d^2}{d^2} \right)^{}  \left( \sum_{ d_1, d_2=1}^{\infty}\frac{c_{d_1} c_{d_2} \gcd(d_1, d_2)}{d_1^2 d_2^2} \right)^{-1/2}.
$$
\end{proof}

\subsection{Numerics.}
It remains to estimate the infinite sums. Note that we expect
$$ \sum_{d=n+1}^{\infty} \frac{c_d^2}{d^2} \sim \frac{1}{n}$$
which suggests that we should get pretty decent approximations by summing up to a reasonably large $n$.
We have
$ \sum_{d=1}^{1000000}c_d^2/ d^2 = 5.60421.$
Assuming that the error decays like $c/n$, we can use finite sums to estimate $c$ by setting
$$  \sum_{d=1}^{10^5}\frac{c_d^2}{d^2}  + \frac{c}{10^5} = \sum_{d=1}^{10^6}\frac{c_d^2}{d^2}  + \frac{c}{10^6} $$
giving $c = 5.2881$ and suggesting a limit of $5.60422$. For the second sum, we expect the error after summing to $n$ to be, recalling the proof of Lemma \ref{lem:8},
$$  \sum_{d_1 = n+1}^{\infty} \sum_{d_2=1}^{\infty}  \frac{c_{d_1} c_{d_2} \gcd(d_1, d_2)}{d_1^2 d_2^2} \lesssim \frac{\log{n}}{n}.$$
This is somewhat comparable to the first sum. Summing
$$ \sum_{ d_1, d_2=1}^{10000}\frac{c_{d_1} c_{d_2} \gcd(d_1, d_2)}{d_1^2 d_2^2} = 10.4912 $$
and using the same type of extrapolation trick, suggesting that the error is $\sim 20.53/N$, suggests the limit to be roughly $10.4933$. This then gives
$$  \lim_{n \rightarrow \infty}  \frac{\left\langle v_n, B_n v_n \right\rangle}{\|v_n\| \cdot \|B_n v_n\|}  \sim 0.997992.$$
We also note that since $A_n$ and $A_n^T$ are sparse and $v_n$ is completely explicit, one can actually test the Theorem for fairly large values of $n$. We were able to run it for $n = 50.000$ and got
$$  \left\langle \frac{v_n}{\|v_n\|}, \frac{A_n^T A_n v_n}{\| A_n^T A_n v_n\|} \right\rangle \sim 0.99754$$
for which the first three digits match.

\end{document}